\documentclass[a4paper,12pt]{amsart}

\usepackage[english]{babel}
\usepackage[utf8]{inputenc}
\usepackage[T1]{fontenc}
\usepackage{amsmath}
\usepackage{amsfonts}
\usepackage{amssymb}
\usepackage{amsthm}
\usepackage{lmodern}

\usepackage{tikz}
\usetikzlibrary{arrows}

\newtheorem{theo}{Theorem}[section]
\newtheorem{prop}[theo]{Proposition}
\newtheorem{lemm}[theo]{Lemma}
\newtheorem{coro}[theo]{Corollary}

\title{Asymptotic values of modular multiplicities for $\operatorname{GL}_2$}
\author{Sandra ROZENSZTAJN}
\address{UMPA, ENS de Lyon\\
UMR 5669 du CNRS\\
46, all\'ee d'Italie\\
69364 Lyon Cedex 07\\
France}
\email{sandra.rozensztajn@ens-lyon.fr}

\begin{document}

\begin{abstract}
We study the irreducible constituents of the reduction modulo $p$
of irreducible algebraic representations $V$ of the group 
$\operatorname{Res}_{K/\mathbb{Q}_p}\operatorname{GL}_2$
for $K$ a finite extension of $\mathbb{Q}_p$. We show that asymptotically,
the multiplicity of each constituent depends only on the dimension of
$V$ and the central character of its reduction modulo $p$. As an application,
we compute the asymptotic value of multiplicities
that are the object of
the Breuil-Mézard conjecture. 
\end{abstract}

\maketitle

\section{Introduction}

\subsection{Main result}

Let $p$ be a prime number and $K$ a finite extension of $\mathbb{Q}_p$ of degree
$h$. We write $H_K = \operatorname{Hom}(K,\overline{\mathbb{Q}}_p)$. 
The irreducible representations of the algebraic group
$G = \operatorname{Res}_{K/\mathbb{Q}_p}\operatorname{GL}_2$ 
are in bijection with families
$(n,m) = (n_{\tau},m_{\tau})_{\tau\in H_K} \in (\mathbb{Z}_{\geq 0}\times \mathbb{Z})^{H_K}$: 
they are the representations
$V_{n,m} = \otimes_{\tau}
(\operatorname{Symm}^{n_{\tau}}W_{\tau}
\otimes \det_{\tau}^{m_{\tau}})$, 
where $W_{\tau}$ is the standard representation of
$\operatorname{GL}_2$ given by the morphism
$\tau : K  \to \overline{\mathbb{Q}}_p$ and $\det_{\tau}
= \det W_{\tau}$. 
We can consider the semi-simplification $\overline{V}_{n,m}^{ss}$
of the reduction in characteristic $p$ of $V_{n,m}$, this is
well-defined and doesn't depend on the choice of a $G$-stable
lattice in $V_{n,m}$ used to define the reduction.
Then $\overline{V}_{n,m}^{ss}$ 
is a representation of the group $\operatorname{GL}_2(\mathbb{F}_q)$ 
(where $q$ is the cardinality of the residue field of $K$). 
In general this representation is not irreducible, and we are interested
in the multiplicities $a_{\sigma}(n,m)$
of the irreducible representations $\sigma$ of
$\operatorname{GL}_2(\mathbb{F}_q)$ that appear as constituents of 
$\overline{V}_{n,m}^{ss}$. They are well defined and do not depend on the choice
of the lattice in $V_{n,m}$. 
In particular we want to study their asymptotic value
as the coefficients $n_{\tau}$ grow to infinity.
We show that the asymptotic value of 
$a_{\sigma}(n,m)$
depends only on $\sigma$ and the dimension and central character of
$\overline{V}_{n,m}^{ss}$.
This follows from the theorem (corollary \ref{coro_asymptote}):

\begin{theo}
\label{theoreme_intro}
There exists a family
$(\mathcal{S}_{\alpha})_{\alpha}$ 
of elements of the Grothen\-dieck group $R$ of representations of  
$\operatorname{GL}_2(\mathbb{F}_q)$ in characteristic $p$, indexed by central characters
$\alpha$ of $\operatorname{GL}_2(\mathbb{F}_q)$, such that
for any representation $W$ of $\operatorname{GL}_2(\mathbb{F}_q)$ that has a central
character, for any representation $U$ of $\operatorname{GL}_2(\mathbb{F}_q)$ that is
the reduction in characteristic $p$ of an irreducible algebraic
representation of 
$\operatorname{Res}_{K/\mathbb{Q}_p}\operatorname{GL}_2$,
the class $[V]$ in $R$ of the representation $V = W \otimes U$
satisfies:
$$
[V] = (\dim V)\mathcal{S}_{\alpha(V)} + O((\dim U)^{1-1/h})
$$
where $\alpha(V)$ denotes the central character of $V$.
\end{theo}

\subsection{Motivation}
\label{motivation}

The motivation for this study comes from a question about Galois
representations that we now explain. We denote the absolute Galois
group of $K$ by $\operatorname{Gal}(\overline{K}/K)$.
We fix a representation 
$\overline{\rho} : \operatorname{Gal}(\overline{K}/K) \to \operatorname{GL}(V)$
where $V$ is a $2$-dimensional vector space over some finite field $k$.
We denote by $F$ a finite extension of $\mathbb{Q}_p$ with residue field $k$
and by $\pi$ a uniformizer of $F$.
In all that follows we assume that
$\operatorname{End}_{k[\operatorname{Gal}(\overline{K}/K)]}(\overline{\rho}) = k$.

Let $R(\overline{\rho})$ be the universal deformation ring of 
$\overline{\rho}$ over $F$. It is an 
$\mathcal{O}_F$-algebra that parameterizes the liftings of 
$\overline{\rho}$ in characteristic $0$. That is, for each closed
point $x$ of  $\operatorname{Spec} R(\overline{\rho})[1/p]$,
there exists a $2$-dimensional vector space $V_x$ over 
some extension of $F$ and a representation
$\rho_x : \operatorname{Gal}(\overline{K}/K) \to \operatorname{GL}(V_x)$
lifting $\overline{\rho}$.

Let now $v = \{(n_{\tau},m_{\tau}),\tau \in H_K\}$ 
be a family of pairs of integers indexed by $H_K$ as before,
let $\mathtt{t}$ be a Galois type, that is a representation of the inertia
subgroup $I_K$ of $\operatorname{Gal}(\overline{K}/K)$
with values in $\operatorname{GL}_2(\overline{\mathbb{Q}}_p)$ and open kernel, 
and let $\psi$ be a character
$\operatorname{Gal}(\overline{K}/K) \to \mathcal{O}_F^*$
satisfying some compatibility relations with $v$ and $\mathtt{t}$
(details
on these relations can be found in Section $2$ of \cite{BrMez2},
but they will play no role in this article).

A representation
$\rho : \operatorname{Gal}(\overline{K}/K) \to \operatorname{GL}_2(L)$ for some extension $L$ of $F$
is said to be of type $(v,\mathtt{t},\psi)$
if $\rho$ is potentially semi-stable
with Hodge-Tate weights
$(m_{\tau},m_{\tau}+n_{\tau}+1)_{\tau \in H_K}$ 
and the restriction to $I_K$
of the Weil-Deligne representation attached to
$D_{pst}(\rho)$
is isomorphic to $\mathtt{t}$
and $\det\rho = \psi$.

Kisin showed in the article 
\cite{Kisin2} (see also \cite{KisinICM})
that, after possibly enlarging $F$, liftings of
$\overline{\rho}$ of type $(v,\mathtt{t},\psi)$
are parameterized by a deformation ring
$R_{v,\mathtt{t}}^{\psi}(\overline{\rho})$ that
is a quotient of $R(\overline{\rho})$: for any closed point
$x$ of $\operatorname{Spec} R(\overline{\rho})[1/p]$, the representation $\rho_x$
is of type $(v,\mathtt{t},\psi)$ if and only if $x$ is a point of the closed
subscheme defined by
$R_{v,\mathtt{t}}^{\psi}(\overline{\rho})$.

The Breuil-Mézard conjecture (\cite{BrMez1})
and its refinements (\cite{BrMez2} and 
\cite{EmGee}, \cite{Kisin1}), proved by Kisin in \cite{Kisin1}
in the case $K = \mathbb{Q}_p$
for most representations $\overline{\rho}$, 
and also by Pa{\v{s}}k{\=u}nas in \cite{Pas},
is a statement about the ring $R_{v,\mathtt{t}}^{\psi}(\overline{\rho})$. 
The goal is to give a measure of its complexity using the Hilbert-Samuel
multiplicity
$e(R_{v,\mathtt{t}}^{\psi}(\overline{\rho})/\pi)$
of its special fiber.
The conjecture states that this multiplicity, called Galois multiplicity,
equals an automorphic multiplicity that can be written as a sum
$\mu_{Aut}(\overline{\rho},v,\mathtt{t}) =
\Sigma_{\sigma}a_{\sigma}(v,\mathtt{t})\mu_{\sigma}(\overline{\rho})$
over all irreducible representations $\sigma$ in characteristic $p$
of $\operatorname{GL}_2(\mathbb{F}_q)$. 
The $a_{\sigma}(v,\mathtt{t})$ are nonnegative integers depending only on
$\sigma$, $v$ and $\mathtt{t}$ and the $\mu_{\sigma}(\overline{\rho})$ are
nonnegative integers, that we call intrinsinc multiplicities following
the terminology of \cite{DM}. They depend only on $\sigma$ and $\overline{\rho}$.  

Let us recall the definition of
$a_{\sigma}(v,\mathtt{t})$: Using the theory of types,
Henniart (in the appendix of \cite{BrMez1}) attaches to $\mathtt{t}$ a
finite-dimensional representation
$\sigma(\mathtt{t})$ of $\operatorname{GL}_2(\mathcal{O}_K)$ with coefficients in $\overline{\mathbb{Q}}_p$.
We attach to
$v = (n_{\tau},m_{\tau})$ a representation $\sigma(v)$
that is the algebraic irreducible representation of
$\operatorname{Res}_{K/\mathbb{Q}_p}\operatorname{GL}_2$ given by
$\otimes_{\tau\in H_K}(\operatorname{Symm}^{n_{\tau}}W_{\tau}\otimes 
\det_{\tau}^{m_{\tau}})^{\tau}$. 
Let $\sigma(v,\mathtt{t}) = \sigma(\mathtt{t}) \otimes \sigma(v)$
and let $T_{v,\mathtt{t}}$ be a lattice in
$\sigma(v,\mathtt{t})$ that is preserved by the action of $\operatorname{GL}_2(\mathcal{O}_K)$. 
We then define 
$a_{\sigma}(v,\mathtt{t})$ 
to be the multiplicity of the irreducible representation 
$\sigma$ of $\operatorname{GL}_2(\mathbb{F}_q)$ in the 
semi-simplification of the reduction to characteristic $p$ 
of the lattice $T_{v,\mathtt{t}}$, which we denote by $\overline{\sigma(v,\mathtt{t})}$.

Thus when we apply the results of Theorem \ref{theoreme_intro}
to $U = \sigma(\mathtt{t})$ and $W = \sigma(v)$ we get an 
asymptotic value of $a_{\sigma}(v,\mathtt{t})$,
and therefore of the automorphic multiplicity (provided we know
the intrinsic multiplicities) and so conjecturally, we get
the asymptotic value of the Hilbert-Samuel multiplicity of the
special fiber of the deformation ring for these Hodge-Tate weights
and this Galois type $\mathtt{t}$. For example,
we treat in Paragraph \ref{ex_unram} the case where
$K$ is the unramified extension of $\mathbb{Q}_p$ of degree $h$, and $\overline{\rho}$
is an irreducible representation that is as generic as possible. 
We obtain the following
asymptotic value for the automorphic multiplicity:
$$
\mu_{Aut}(\overline{\rho},v,\mathtt{t}) = 
\frac{4^h}{p^{2h}-1}\dim \sigma(\mathtt{t})(\prod_{i=0}^{h-1}(n_i+1))
+ O(\prod_{i=0}^{h-1}(n_i+1))^{1-1/h})
$$
for values of $v = (n_i,m_i)_{0\leq i\leq h-1}$ such that
the central character of $\overline{\sigma(v,\mathtt{t})}$ has some
fixed value depending on $\overline{\rho}$, and $\mu_{Aut}(\overline{\rho},v,\mathtt{t}) = 0$
otherwise.

\subsection{Plan of the article}

The plan of this note is as follows:
in Section \ref{section_representations}, we recall some results on
irreducible representations of 
$\operatorname{GL}_2(\mathbb{F}_q)$ in characteristic $p$ and the multiplicities
of these representations in principal series representations.
In Section \ref{section_asymptote} we compute asymptotic values of
the multiplicities of the irreducible representations of $\operatorname{GL}_2(\mathbb{F}_q)$,
in the reduction in characteristic $p$ of irreducible algebraic
representations of  $\operatorname{Res}_{K/\mathbb{Q}_p}\operatorname{GL}_2$ in characteristic $0$, and
we prove Theorem \ref{theoreme_intro}. In Section
\ref{section_application} we apply this result to the computation of
asymptotic values, as Hodge-Tate weights go to infinity, of automorphic
multiplicities for some mod $p$ Galois representations.

\subsection{Acknowledgements} I would like to thank Christophe Breuil
for useful discussions and helpful comments.

\section{Representations of $\operatorname{GL}_2(\mathbb{F}_q)$ in characteristic $p$}
\label{section_representations}

\subsection{Notations}
\label{notations}

In the following we fix a prime number $p$ and $q = p^f$.
We denote by $E$ an algebraic closure of $\mathbb{F}_q$ and we choose
a morphism $\mathbb{F}_q \to E$. This allows us to view $\mathbb{F}_q$ as a subfield of
$E$.

We denote $\operatorname{GL}_2(\mathbb{F}_q)$ by $G$ and the Borel subgroup
of upper triangular matrices by $B$.

\subsection{Irreducible representations}
\label{repr_L_et_S}

We recall the description of the irreducible characteristic $p$
representations of $G$. 
In this case everything can be described 
in a very explicit manner.
The proofs of the statements of this Paragraph 
can be found in Chapter 10 of the book
\cite{Bon} (the results in \cite{Bon} are for the case of $\operatorname{SL}_2(\mathbb{F}_q)$
but they extend readily to the case of $\operatorname{GL}_2(\mathbb{F}_q)$). The description of
the irreductible representations can also be found in Section 1 of
\cite{BL}.

We denote by $\mathcal{W}$ the standard representation of $G$ on $E^2$. For
all $n \in \mathbb{Z}_{\geq 0}$ we let $S_n = \operatorname{Symm}^n(\mathcal{W})$. The representation
$S_n$ is of dimension $n+1$ and can be identified with the vector space
$E[x,y]_n$ of homogeneous polynomials of degree $n$ with the action
of  
$g =  \left(\begin{smallmatrix} a & b \\ c & d \end{smallmatrix}\right)$ 
given by $(g\cdot P)(x,y) = P(ax+cy,bx+dy)$.
In particular $S_0$ is the trivial representation of $G$.

Let $(\rho,V)$ be a finite-dimensional $E$-representation of $G$.
We denote by $V(i)$ the twist of $V$ by the $i$-th power of the
determinant for an $i \in \mathbb{Z}/(q-1)\mathbb{Z}$. For any integer $j$ we denote by 
$V^{[j]}$ the representation $\rho \circ F^j$, where $F$ is the
Frobenius of $\mathbb{F}_q$. Note that $V^{[f]} = V$.

We now describe the irreducible representations of $G$
(Theorem 10.1.8 of \cite{Bon}):
Let $n$ be an integer written in base $p$ as
$\sum_{i \geq 0} n_ip^i$. We denote by $L_n$
the sub-representation of $S_n = E[x,y]_n$ generated by $x^n$. 
It is the subspace 
generated by the monomials $x^my^{n-m}$ for all integers $m$
that can be written in base $p$ as
$\sum_{i\geq 0}m_ip^i$ with $0 \leq m_i \leq n_i$ for all $i$.
As a $G$-representation, 
$L_n$ is isomorphic to $\otimes_{i\geq 0} S_{n_i}^{[i]}$.

\begin{prop}
The irreducible representations of $G$ are exactly the
representations
$L_n(m)$ with $0 \leq n \leq q-1$ and $m \in \mathbb{Z}/(q-1)\mathbb{Z}$. 
\end{prop}

Moreover the following property holds (Proposition 10.1.18 and Paragraph
10.2.1 of \cite{Bon}):

\begin{prop}
\label{LS}
For any $0 \leq n \leq q-1$, $L_n$ is a multiplicity one constituent
of $S_n$. The other irreducible constituents in $S_n$ are among the
representations $L_i(j)$ with $i < n$.
\end{prop}

\subsection{Frobenius action and $\theta$ operators}

The Frobenius acts as a permutation on the set 
$\{L_n(m),0\leq n \leq q-1,m \in \mathbb{Z}/(q-1)\mathbb{Z}\}$. 
In order to describe this permutation
we introduce an operator on 
$\{0,\dots,q-1\}$ and $\mathbb{Z}/(q-1)\mathbb{Z}$ that we denote by $\theta$. 
On $\mathbb{Z}/(q-1)\mathbb{Z}$, $\theta$ is simply multiplication by $p$.
Let $n$ be an element of 
$\{0,\dots,q-1\}$. We write it in base $p$ as
$n = \sum_{i=0}^{f-1} a_ip^i$
and we set
$\theta n = a_{f-1} + a_0p + \dots a_{f-2}p^{f-1}$.
These $\theta$ operators are compatible with respect
to reduction modulo $q-1$.
We now have $L_n(m)^{[1]} = L_{\theta n}(\theta m)$. 

\subsection{The Grothendieck group of representations of $G$
in characteristic $p$}

Let $\operatorname{Rep}_E(G)$ be the Grothendieck group of finite-dimensional
representations of $G$ with coefficients in $E$. We denote by $[V]$ the
image in $\operatorname{Rep}_E(G)$ of a representation $V$ of $G$.

$\operatorname{Rep}_E(G)$ is a ring with the product given by $[V \otimes W] = [V][W]$. 
The twist by a power of the determinant and the action of Frobenius extend
to $\operatorname{Rep}_E(G)$, and we denote $[V(i)]$ by $[V](i)$ and
$[V^{[j]}]$ by $[V]^{[j]}$. 

We can endow $\operatorname{Rep}_E(G)$ with two natural bases: 
The first one, which we call basis $L$, is given by the irreducible
representations of $G$, that is the 
$[L_n(m)]$ for $0 \leq n \leq q-1$ and $m \in \mathbb{Z}/(q-1)\mathbb{Z}$.
Note that tensoring an element of $L$ by a power of the determinant or
twisting it by a power of Frobenius transforms this element into another
element of $L$.
The second one, which we call basis $S$, is given by the 
$[S_n(m)]$ for $0\leq n \leq q-1$ and $m \in \mathbb{Z}/(q-1)\mathbb{Z}$.
That it is indeed a basis follows from Proposition \ref{LS}.

We denote by $[V : L_n(m)]$ the coefficient of $[L_n(m)]$
when we write an element $V \in \operatorname{Rep}_E(G)$ in the basis $L$.

Finally let 
$R = \operatorname{Rep}_E(G)\otimes_{\mathbb{Z}}\mathbb{R}$.
All the notations that we have just defined extend to $R$.

\subsection{Recurrence formula for tensor products}

The following formula holds for the representations
$S_n = \operatorname{Symm}^n(\mathcal{W})$ of $G$:

\begin{prop}
\label{produittens}
For all nonnegative $n,m$ we have
$[S_n][S_m] = [S_{n+m}]+[S_{n-1}][S_{m-1}](1)$.
\end{prop}

\begin{proof}
We have in fact a more precise result which is proved in Glover's article
\cite{Glo}:
the following is an exact sequence:
$$
0 \to (S_{n-1}\otimes S_{m-1})(1) \to  
S_n\otimes S_m \to S_{n+m} \to 0,
$$
where the second arrow is given by 
multiplication of polynomials and the first one
is given by
$u\otimes v \mapsto (xu\otimes yv)-(yu\otimes xv)$.
\end{proof}

\subsection{The Dixon invariant}

Let $\Theta = xy^q - x^qy$. For all 
$g\in \operatorname{GL}_2(\mathbb{F}_q)$ we have
$g(\Theta) = (\det g)\Theta$.
For all $n \in \mathbb{Z}_{\geq 0}$ we can consider
the map
$\iota : S_n(1) \to S_{n+q+1}$
given by multiplication by $\Theta$.
This map is 
$\operatorname{GL}_2(\mathbb{F}_q)$-equivariant and injective.
We let
$\mathcal{V}_{n+2} = S_{n+q+1}/\iota S_n(1)$.

\subsection{Principal series representations}

\subsubsection{Definition}

Let $\mathcal{F}$ be the space of functions on $G$ with values in $\mathbb{F}_q$.
The group $G$ acts on $\mathcal{F}$ by $(gf)(\gamma) = f(\gamma g)$.
If $\chi$ is a character of $B$, we let $\mathcal{V}(\chi) = \operatorname{ind}_B^G\chi$.
It is the subspace of $\mathcal{F}$ formed by the functions satisfying
$f(bg) = \chi(b)f(g)$ for all $b\in B$ and $g\in G$. It is
an $\mathbb{F}_q$-vector space of dimension $q+1$.
The dual space $\mathcal{V}(\chi)^*$ of $\mathcal{V}(\chi)$ is isomorphic to
$\mathcal{V}(\chi^{-1})$.

For $r \in \mathbb{Z}$, we denote by $\lambda_r$ the character of $B$ with values
in $\mathbb{F}_q^*$ given by $\begin{pmatrix} a & b \\ 0 & d
\end{pmatrix}\mapsto d^r$. This character depends
only on $r$ modulo $q-1$. 
Let us note that all characters of $B$ with values in 
$\mathbb{F}_q^*$ are of the form $\lambda_r \otimes \det^i$. We will thus
restrict our study to the case of the characters $\lambda_r$.

\subsubsection{Relation to $[S_n(r)]$}

Let $n \geq 0$ be an integer. 
We define a $G$-equivariant morphism
$u_n : S_n = E[x,y]_n \to \mathcal{F}$ by $P \mapsto (g \mapsto P((0,1)g))$.
Its image is contained in $\mathcal{V}(\lambda_n)$. We have the following result:

\begin{lemm}
\label{iso_sp_sym}
Let $r \in \mathbb{Z}/(q-1)\mathbb{Z}$ and let $n \geq 0$ be an integer congruent to
$r$ modulo $q-1$. 
If $n \leq q$ then $u_n$ is an embedding
$S_n \to \mathcal{V}(\lambda_r)$, and it is an isomorphism if $n = q$. 
If $n \geq q+1$, then $u_n$ induces an isomorphism
$\mathcal{V}_{n-q+1} =  S_n/\iota S_{n-q-1}(1) \to \mathcal{V}(\lambda_r)$.
\end{lemm}

\begin{proof}
The kernel of $u_n$ is the set of homogenous polynomials of
degree $n$ that are zero on  $\mathbb{F}_q^2$, that is, multiples of the
Dixon invariant $\Theta$. In particular, $u_n$ is injective when $n <
q+1$, and the kernel of $u_n$ is $\iota S_{n-q-1}(1)$ otherwise. 
Hence when $n \geq q+1$ the morphism $u_n$ induces an embedding of 
$\mathcal{V}_{n-q+1} =  S_n/\iota
S_{n-q-1}(1)$ into $\mathcal{V}(\lambda_r)$. As both spaces are of dimension
$q+1$, this is an isomorphism.
\end{proof}

In particular:
\begin{lemm}
For $n \geq 2$,
$\mathcal{V}_n$ and $\mathcal{V}(\lambda_n)$ are isomorphic and the isomorphism class
of $\mathcal{V}_n$ depends only on the class of $n$ modulo $q-1$.
\end{lemm}

In the following we will use the notation $\mathcal{V}_n$ for $\mathcal{V}(\lambda_n)$
also in the case where $n$ is an element of $\mathbb{Z}/(q-1)\mathbb{Z}$ and not
necessarily an integer $\geq 2$.

We can then recover a classical result (that we won't use in the sequel)
using the fact that $S_n^*$ is isomorphic to $S_n(-n)$ and that
$S_n$ and $S_{q-1-n}(n)$ have no common constituent
if $n \leq q-1$ (see Proposition \ref{Diamond}):

\begin{prop}
\label{decomp_sp}
For all $0 \leq n \leq q-1$ there is 
an exact sequence
$0 \to S_n \to \mathcal{V}(\lambda_n) \to S_{q-1-n}(n) \to 0$. It is
split when $n = 0$ or ${q-1}$.
\end{prop}

\subsubsection{Constituents of principal series representations}

The constituents of the representations $[\mathcal{V}_r]$ have been computed by
Diamond in the article \cite{Dia}. We give here a reformulation of this
description.
Consider the following graph:

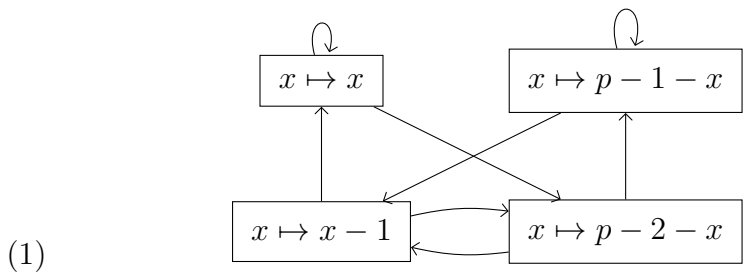
\begin{equation}
\label{graphe1}
\begin{tikzpicture}[->,>={angle 90},
every node/.style={rectangle,draw,inner sep=7pt}]
\node (A) at (0,2) {$x \mapsto x$}; 
\node (B) at (4,2) {$x \mapsto p-1-x$}; 
\node (C) at (0,0) {$x \mapsto x-1$}; 
\node (D) at (4,0) {$x \mapsto p-2-x$};
\path
(A) edge [loop above] (A)
(B) edge [loop above] (B)
(C) edge (A)
(A) edge (D)
(B) edge (C)
(D) edge (B)
(C) edge [bend left=10] (D)
(D) edge [bend left=10] (C);
\end{tikzpicture}
\end{equation}

Let $\mathcal{C}_f$ be the set of closed paths of length $f$
in graph (\ref{graphe1}). 
Let $c$ be such a path.
We denote by $c_i$ the vertex reached at  the $i$-th step of the path.
We thus get a list $c_0,\dots,c_{f-1},c_f = c_0$ of vertices of the graph.

To each path $c$ we attach a family of functions
$(\lambda^c_i)_{0 \leq i \leq f-1}$ by taking for
$\lambda^c_i$ the function written on the vertex $c_i$.
We also define a function $\ell_c$, by
$\ell_c(x_0,\dots,x_{f-1}) = 
(\sum_ip^i(x_i - \lambda^c_i(x_i)))/2$ if the vertex 
$c_{f-1}$ is in the left column of the graph,
and  $\ell_c(x_0,\dots,x_{f-1}) = 
(p^f-1 + \sum_ip^i(x_i-\lambda^c_i(x_i))/2$ if the vertex $c_{f-1}$ 
is in the right column.

Let $0 \leq n < q-1$. We write $n$ in base $p$ as 
$n = \sum_{i=0}^{f-1}p^in_i$.
Let $\lambda^c(n) = \sum_{i=0}^{f-1}p^i\lambda^c_i(n_i)$, and
$\ell_c(n) = \ell_c(n_0,\dots,n_{f-1})$. 
Observe that $\lambda^c(n) = \lambda^{c'}(n)$
and $\ell_c(n) = \ell_{c'}(n)$ if and only if $c = c'$.

Diamond's result is the following:
\begin{prop}
\label{Diamond}
$$
[\mathcal{V}_n(m)] = \sum [L_{\lambda^c(n)}(m+\ell_c(n))]
$$
where the sum is over all $c \in \mathcal{C}_f$ 
such that $0 \leq \lambda^c_i(n_i) \leq p-1$ for all $i$.
In particular, all constituents appear with multiplicity $1$.
\end{prop}

\subsubsection{Representations having $[L_n(m)]$ as a constituent}
\label{apparition}

Let $A(n,m)$ be the set of $(n',m') \in \{0,\dots,q-2\}\times \mathbb{Z}/(q-1)\mathbb{Z}$
such that $L_n(m)$ is a constituent of $\mathcal{V}_{n'}(m')$. 
We give a description of $A(n,m)$.
Consider the following graph:

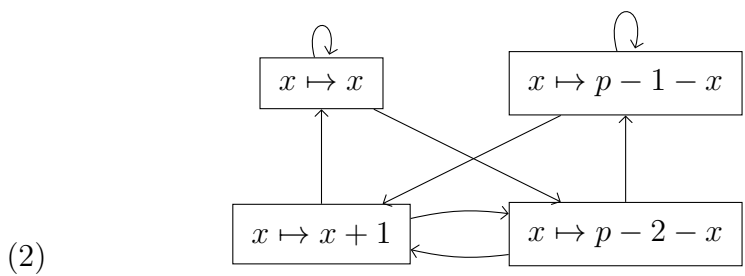
\begin{equation}
\label{graphe2}
\begin{tikzpicture}[->,>={angle 90},
every node/.style={rectangle,draw,inner sep=7pt}]
\node (A) at (0,2) {$x \mapsto x$}; 
\node (B) at (4,2) {$x \mapsto p-1-x$}; 
\node (C) at (0,0) {$x \mapsto x+1$}; 
\node (D) at (4,0) {$x \mapsto p-2-x$};
\path
(A) edge [loop above] (A)
(B) edge [loop above] (B)
(C) edge (A)
(A) edge (D)
(B) edge (C)
(D) edge (B)
(C) edge [bend left=10] (D)
(D) edge [bend left=10] (C);
\end{tikzpicture}
\end{equation}

Let $\mathcal{A}_f$ be the set of closed paths of length
$f$ in graph (\ref{graphe2}). 
Let $c$ be such a path and write $c_i$ for the vertex reached
at step $i$. We attach to $c$ a family of functions
$(\mu^c_i)_{0 \leq i \leq f-1}$ where
$\mu^c_i$ is the function written on the vertex $c_i$.

We say that $c$ is compatible with $n = \sum_i n_ip^i$ if
$0 \leq \mu^c_i(n_i) \leq p-1$
for all $i$. In this case let $\mu^c(n) = \sum \mu^c_i(n_i)p^i$

\begin{prop}
\label{antecedents}
The set $A(n,m)$ is the set of $(\mu^c(n),m-\ell_c(\mu^c(n)))$, 
for $c$ in the set of
elements of $\mathcal{A}_f$ that are compatible with $n$ 
and such that $\mu^c(n) \neq q-1$. 
Distinct $c$'s give distinct elements in $A(n,m)$.
The cardinality of $A(n,m)$ doesn't depend on $m$.
\end{prop}

\begin{proof}
We denote by the same letter $c$ a path in graph (\ref{graphe2})
and the analogous path in graph (\ref{graphe1}). Then $\mu^c$ and
$\lambda^c$ are inverse functions, so that $\lambda^c(n') = n$ if
and only if $\mu^c(n) = n'$. Hence the description of $A(n,m)$. In
particular, the cardinality of $A(n,m)$ doesn't depend on $m$.
The fact that distinct $c$'s give distinct elements in $A(n,m)$
follows from the fact that 
$\lambda^c(n) = \lambda^{c'}(n)$
and $\ell_c(n) = \ell_{c'}(n)$ if and only if $c = c'$.
\end{proof}

We denote by $\omega(n)$ the cardinality of $A(n,m)$.
Let $0 \leq n \leq q-1$ be written as $\sum_i n_ip^i$ in base $p$. 
We denote by $r_n$ the number of coefficients $n_i$ that are
equal to $p-1$. Then :

\begin{prop}
\label{calcul_omegan}
\begin{enumerate}
\item
$\omega(0) = 2^f-1$
\item
if $n \neq 0$ then $\omega(n) = 2^{f-r_n}$
\end{enumerate}
\end{prop}

\begin{proof}
Computing $\omega(n)$ is the same as counting elements $c \in
\mathcal{A}_f$ that are compatible with $n$ 
and such that $\mu^c(n) \neq q-1$. 

We have to understand for which cases we have $\mu^c(n) = q-1$.
This can happen only if path $c$ never reaches the vertex labelled
$x \mapsto p-2-x$, so $c$ is necessarily a constant path staying either
on the vertex $x \mapsto x$ or on the vertex $x \mapsto p-1-x$.
For such a $c$ we have $\mu^c(n)=q-1$ if and only if $n = 0$ or $n=q-1$.

For each pair of vertices of graph
(\ref{graphe2}) there is exactly one path of length $2$ from the first
point to the second one. Hence the cardinality of $\mathcal{A}_f$ is
$2^f$: we have $4$ choices for the first vertex of path $c$, $2$ choices
for each following step except for step $f-1$ where we have only one
possibility that allows us to close the path at step $f$.

If $r_n < f$ there exist exactly $2^{f-r_n}$ elements of
$\mathcal{A}_f$ that are compatible with $n$. 
Indeed we can assume without loss of generality that $n_{f-1} < p-1$.
Then for all $i$ such that $n_i = p-1$ the number of possible paths is
divided by $2$, as the path cannot reach the lower row of graph
(\ref{graphe2}) at step $i$.

If $n = q-1$, the paths compatible with $n$ are those that avoid
the lower row of graph \ref{graphe2}. Hence there are only $2$ compatible
paths: the constant paths staying either on the vertex labelled
$x \mapsto x$ or on the vertex labelled $x \mapsto p-1-x$. 
The first one gives $\mu^c(q-1) = q-1$. So we get $\omega(q-1) = 1$. 

\end{proof}

\section{Asymptotic study of representations}
\label{section_asymptote}

\subsection{Central characters}

Let $Z$ be the center of $G$. 
Using the embedding $\mathbb{F}_q \to E$ that we fixed in Paragraph \ref{notations},
we can identify the group of characters of $Z$ to $\mathbb{Z}/(q-1)\mathbb{Z}$ so that
$i \in \mathbb{Z}/(q-1)\mathbb{Z}$ corresponds to $\begin{pmatrix} a & 0
\\ 0 & a \end{pmatrix}\mapsto a^i$.

Let $V$ be a representation of $G$. If is has a central character, we denote it
by $\alpha(V)$.

\subsection{The family $(\mathcal{S}_{\alpha})$}

Let $\mathcal{S} = \frac{1}{q^2-1}\sum_{\chi : B \to \mathbb{F}_q^*}[\mathcal{V}(\chi)] \in R$.
Let $i$ be an element of $\mathbb{Z}/(q-1)\mathbb{Z}$ which we see as a character of $Z$.
Let
$\mathcal{S}_i = \frac{1}{q^2-1}\sum_{\chi : B \to \mathbb{F}_q^*, \chi_{|Z} = i}[\mathcal{V}(\chi)]$.
We have normalized $\mathcal{S}_i$ in such a way that $\dim \mathcal{S}_i = 1$, 
and by construction
$\alpha(\mathcal{S}_i) = i$.

Note that $\mathcal{V}(\chi)^{[j]} = \mathcal{V}(p^j\chi)$, and 
$\alpha(p^j\chi) = p^j\alpha(\chi)$. 
Thus $\mathcal{S}_i^{[j]} = \mathcal{S}_{\theta^ji}$.
We also have $\mathcal{S}_i(j) = \mathcal{S}_{i+2j}$.

\subsection{Norms on $R$}
\label{norms}

We endow the $\mathbb{R}$-algebra
$R = \operatorname{Rep}_E(G)\otimes_{\mathbb{Z}}\mathbb{R}$
with a norm $\|\cdot\|$
that satisfies the following properties:
\begin{enumerate}
\item
it is an algebra norm: that is, for all $V$ and $W$, we have
the inequality
$\| V \cdot W \| \leq \| V\| \| W\|$
\item
it is unchanged by twist by the determinant
\item
it is unchanged under the action of Frobenius
\item
for all $V \in R$ and  $n \in \{0,\dots,q-1\}$ and
$m \in \mathbb{Z}/(q-1)\mathbb{Z}$, we have $|[V:L_n(m)]| \leq \|V\|$
\end{enumerate}

Such a norm can be constructed as follows: first endow $R$ with the
norm $|\cdot\nolinebreak |_{L,\infty}$ defined by
$|V|_{L,\infty} = \max_{n,m}|[V:L_n(m)]|$.
This norm is unchanged by twist by the determinant and action of
Frobenius as these actions permute the elements of the basis $L$.
For all $V \in R$ we set 
$$
\|V\| =
\sup_{W\neq 0}|V\cdot W|_{L,\infty}/|W|_{L,\infty}
$$ 
which is finite as $R$ is a finite dimensional $\mathbb{R}$-vector space and
multiplication by $V$ is $\mathbb{R}$-linear.
Then 
$\|\cdot\|$ is a norm on $R$ that clearly satisfies the first three
properties. That it also satisfies the last one can be seen by noticing
that $\|V \| \geq |V\cdot L_0|_{L,\infty} = |V|_{L,\infty}$. 

We also need another norm on $R$ that we denote by $|\cdot |_{S,1}$.
An element $V$ of $R$ can be written
$V = \sum_{n,m} \lambda_{n,m}[S_n(m)]$ in the basis $S$. We then set
$|V|_{S,1} = \sum_{n,m} |\lambda_{n,m}|$.

As $R$ is a finite-dimension $\mathbb{R}$-vector space, all norms on it are
equivalent. Hence there exists a positive real number 
$M$ such that for all $V \in R$
we have $|V|_{S,1} \leq M\|V\|$.

\subsection{Statement of the results}

We devote the rest of Section
\ref{section_asymptote}
to the proof of the following theorem and its corollary:

\begin{theo}
\label{asymptote}
Let $r \geq 1$ be an integer. There exists a constant $C_r$ depending
only on $r$, such that for any representation $W$ of $G$ with a central
character, and any representation $U$ of the form
$U = S_{k_1}^{[j_1]} \otimes \dots \otimes S_{k_r}^{[j_r]}$,
the representation $V = W\otimes U$ satisfies
$[V] = (\dim V)\mathcal{S}_{\alpha(V)} + r_V$ with $\|r_V\|
\leq C_r|W|_{S,1}(\dim U)/\min_i(k_i+1)$.
\end{theo}

\begin{coro}
\label{coro_asymptote}
There exists a constant $C$ depending only on $K$ such that for any
representation $W$ of $G$ with a central character, and any
representation
$U$ of $G$ that is the reduction in characteristic $p$ of an irreducible
algebraic representation of $\operatorname{Res}_{K/\mathbb{Q}_p}\operatorname{GL}_2$, the representation
$V = W \otimes U$
satisfies $[V] = (\dim V)\mathcal{S}_{\alpha(V)} + r_V$ with
$\|r_V\| \leq C\|W\|(\dim U)^{1-1/h}$.
\end{coro}

\subsection{Computation of $[S_k]$}

Let $k = u(q^2-1) + v(q+1) + w$, where $u = \lfloor k/(q^2-1) \rfloor$,
and where $v(q+1)+w$ is the remainder 
of the euclidean division of $k$ by $q^2-1$,
and $w$ is the remainder of the division of $k$ by $q+1$.

Making use of the equality $[S_{n+q+1}] = [S_n](1) + [\mathcal{V}_{n+2}]$ we get
$[S_k] = \sum_{i=1}^{u(q-1)+v}[\mathcal{V}_{k-i(q+1)+2}](i-1) + [S_w](u(q-1) + v)$.

As $[\mathcal{V}_n]$ is $(q-1)$-periodic we have:
$$
\sum_{i=u(q-1)+1}^{u(q-1)+v}[\mathcal{V}_{k-i(q+1)+2}](i-1) = 
\sum_{i=1}^{v}[\mathcal{V}_{k-i(q+1)+2}(i-1)]
= \sum_{i=0}^{v-1}[\mathcal{V}_{k-2i}(i)]
$$
and
$$
\sum_{i=1}^{u(q-1)}[\mathcal{V}_{k-i(q+1)+2}](i-1) = 
u \sum_{i=1}^{q-1}[\mathcal{V}_{k-i(q+1)+2}(i-1)]
= u \sum_{i=0}^{q-2}[\mathcal{V}_{k-2i}(i)]
$$

Moreover
$[\sum_{i=0}^{q-2}V_{k-2i}(i)] =
(q^2-1)\mathcal{S}_{(k \bmod q-1)}$.
Let 
$$
r_k = (u(q^2-1)-k-1)\mathcal{S}_{(k \bmod q-1)}
+ \sum_{i=0}^{q-2}[\mathcal{V}_{k-2i}(i)]
+ [S_w](u(q-1)+v)
$$

We define the constant $A$ as follows:
$$
A = (q^2+2q)\max \{ \|[S_r]\|, 0 \leq r  < q^2-1, \|\mathcal{S}_i\|, i \in \mathbb{Z}/(q-1)\mathbb{Z} \}
$$
As $|u(q^2-1)-k-1| \leq q^2$ we get:

\begin{prop}
\label{equivSk}
$[S_k] = (k+1)\mathcal{S}_{(k \bmod q-1)} + r_k$ and $\|r_k\| \leq A$.
\end{prop}

\subsection{Action of Frobenius}

\begin{prop}
\label{frob}
$[S_k]^{[j]} = (k+1)\mathcal{S}_{\theta^j (k \bmod q-1)} + r_k^{[j]}$ 
and  $\|r_k^{[j]}\| \leq A$. 
\end{prop}

\begin{proof}
From $[S_k] = (k+1)\mathcal{S}_{(k \bmod q-1)} + r_k$ we get the equality
$[S_k]^{[j]} = (k+1)\mathcal{S}_{(k \bmod q-1)}^{[j]} + r_k^{[j]}$. 
We know that $\mathcal{S}_i^{[j]} = \mathcal{S}_{\theta^ji}$, and moreover
$\| r_k \| = \|r_k^{[j]}\|$.
\end{proof}

\begin{coro}
\label{proche_frobenius}
Let $t_{j,k}$ be an integer such that $2t_{j,k} = \theta^jk-k \bmod q-1$.
Then $\|[S_k]^{[j]} - [S_k](t_{j,k})\| \leq 2A$.
\end{coro}

\begin{proof}
This follows from Proposition \ref{frob} as soon as we have proved that
such a $t_{j,k}$ exists. When $p=2$, multiplication by $2$ is an
automorphism of $\mathbb{Z}/(q-1)\mathbb{Z}$, hence $t_{j,k}$ exists (and is unique).
When $p>2$, $\theta^jk$ has the same parity as $k$, as $\theta$ is the
same as multiplication by $p$ modulo $q-1$, so $\theta^jk-k$ is even
which gives the existence of $t_{j,k}$. 
\end{proof}

\subsection{Product}

\begin{prop}
\label{tensoriel_s}
Let $a,b \in \mathbb{Z}/(q-1)\mathbb{Z}$. Then $\mathcal{S}_a  \mathcal{S}_b = \mathcal{S}_{a+b}$.
\end{prop}

\begin{proof}
Let $u,v \to +\infty$  with $u = a \pmod {q-1}$,
$v = b \pmod {q-1}$ and $u \leq v$. 

We have the following equality:
\begin{eqnarray*}
[S_u][S_v] &=& ((u+1)\mathcal{S}_{(u \bmod q-1)} + O(1))((v+1)\mathcal{S}_{(v \bmod q-1)}+ O(1)) \\
	&=& (u+1)(v+1)\mathcal{S}_{(u \bmod q-1)}\mathcal{S}_{(v \bmod q-1)}+ O(\max(u,v))
\end{eqnarray*}

Moreover:
$$
[S_u][S_v] = (u+1)(v+1)\mathcal{S}_{(u+v \bmod q-1)} + O(\min(u,v))
$$
Indeed it follows from \ref{produittens} that
$[S_u][S_v] = [S_{u+v}] + [S_{u-1}][S_{v-1}](1)$, so
$[S_u][S_v] = \sum_{i=0}^u[S_{u+v-2i}](i)$.
Hence:
\begin{eqnarray*}
[S_u][S_v] &=& \sum_{i=0}^u((u+v-2i+1)\mathcal{S}_{(u+v \bmod q-1)} + r_{u+v-2i}) \\
		&=& (u+1)(v+1)\mathcal{S}_{(u+v \bmod q-1)} + \sum_{i=0}^ur_{u+v-2i}
\end{eqnarray*}
Finally $\|[S_u][S_v] - (u+1)(v+1)\mathcal{S}_{(u+v \bmod q-1)}\| \leq (\min(u,v)+1)A$.

The result then follows from the comparison of the two formulas for
$[S_u][S_v]$.
\end{proof}

\subsection{Proof of the theorem}
\label{preuve}

We can now prove Theorem \ref{asymptote} and its corollary.

\begin{proof}[Proof of Theorem \ref{asymptote}]
Using linearity and the invariance of the norm with respect to twisting
by the determinant, we need only prove it when $W = S_m$ and 
$0 \leq m <q$.

Suppose first that $r=1$. We write $k,j$ for $k_1,j_1$.
Let $V = W \otimes S_k^{[j]}$ as in the statement of the theorem.

Let $t$ be an integer as in corollary \ref{proche_frobenius}, so that
$[S_k^{[j]}] = [S_{k}](t) + a$ for some $a$ with $\|a \| \leq 2A$.
Then $[W \otimes S_k^{[j]}] = [S_m][S_{k}](t) + [W]a$.
By the computations in the proof of Proposition
\ref{tensoriel_s} we see that
$$
[S_m][S_k] = (m+1)(k+1)\mathcal{S}_{(m+k \bmod q-1)} + b
$$
for
some $b\in R$ with $\|b\| \leq qA$. Finally 
$$
[W \otimes S_k^{[j]}] 
= (\dim (W\otimes S_k^{[j]}))\mathcal{S}_{\alpha(W\otimes S_k^{[j]})} + c
$$
where
$c = b(t) + [W]a$, so that $\|c\| \leq qA+2A^2$. We get the result
with $C = qA+2A^2$.

Let us now return to the general case.
Let $U = S_{k_1}^{[j_1]} \otimes \dots \otimes S_{k_r}^{[j_r]}$
and $V = W \otimes U$.

We set $V_1 = W \otimes S_{k_1}^{[j_1]}$ and $V_i = S_{k_i}^{[j_i]}$
for $i \geq 2$.
Then $[V] = \prod_{i=1}^r[V_i]$.

For all $i \geq 2$ we have $[V_i] = (\dim V_i)\mathcal{S}_{\alpha(V_i)} + c_i$ with
$\|c_i \| \leq A$ by Proposition \ref{frob}. 
For $i=1$ we have $[V_1] = (\dim V_1)\mathcal{S}_{\alpha(V_1)} + c_1$ with $\|c_1\| \leq
C_1 = A(q+2A)$ by the case $r=1$ of the theorem.
Moreover $\dim V_1 \leq q(k_1+1)$, as $\dim W \leq q$.

Let $r_V = [V] - (\dim V)\mathcal{S}_{\alpha(V)}$.
We have:
$$
[V] = \prod_{i}((\dim V_i)\mathcal{S}_{\alpha(V_i)}+c_i)
= \sum_{I \subset \{1,\dots,r\}}(\prod_{i\in I}(\dim
V_i)\mathcal{S}_{\alpha(V_i)})(\prod_{i\not\in I}c_i).
$$
Using Proposition \ref{tensoriel_s} 
we see that $(\dim V)\mathcal{S}_{\alpha(V)} = \prod_i(\dim V_i)\mathcal{S}_{\alpha(V_i)}$,
hence:
$$
r_V = \sum_{I \subset \{1,\dots,r\}, I \neq \{1,\dots,r\}}
(\prod_{i\in I}(\dim V_i)\mathcal{S}_{\alpha(V_i)})(\prod_{i\not\in I}c_i).
$$
We have:
$$
\|(\prod_{i\in I}\mathcal{S}_{\alpha(V_i)})(\prod_{i\not\in I}c_i)\| \leq
\prod_{i\in I}\|\mathcal{S}_{\alpha(V_i)}\|\prod_{i\not\in I}\|c_i\| \leq
A^{r}(2A+q)
$$ 
as for each $i$ we have 
$\|\mathcal{S}_{\alpha(V_i)}\| \leq A$ by definition of $A$, and $\|c_i\|\leq A$
except for $\|c_1\| \leq A(2A +q)$.
Finally:
$$
\|r_V\| \leq A^r(2A+q) \sum_{I \subset \{1,\dots,r\}, I \neq
\{1,\dots,r\}} \prod_{i\in I}\dim V_i.
$$
Moreover 
$\prod_{i\in I}\dim V_i \leq q(\dim U)/\min_{1\leq i \leq r}(k_i+1)$.
Hence:
$$
\|r_V\| \leq q2^rA^r(2A+q)(\dim U)/\min_i(k_i+1).
$$
The result follows with $C_r = q2^rA^r(2A+q)$.
\end{proof}

\begin{proof}[Proof of Corollary \ref{coro_asymptote}]

The irreducible representations of
$\operatorname{Res}_{K/\mathbb{Q}_p}\operatorname{GL}_2$
are of the form
$\otimes_{\tau\in \operatorname{Hom}(K,\overline{\mathbb{Q}}_p)}
(\operatorname{Symm}^{k_{\tau}}W_{\tau}\otimes \det_{\tau}^{j_{\tau}})$
for some $k_{\tau} \in \mathbb{Z}_{\geq 0}$
and $j_{\tau} \in \mathbb{Z}$. 
Let $U$ be the reduction in characteristic $p$ of such a
representation. Then $[U]$ is of the form
$[U] = [S_{k_1}^{[j_1]}\otimes\dots \otimes S_{k_h}^{[j_h]}(n)]$ 
where $h = [K : \mathbb{Q}_p]$. 

Let $I$ be a nonempty subset of $\{1,\dots,h\}$ and let $r$ be the
cardinality of $I$.
We apply Theorem
\ref{asymptote}
to the decomposition
$V = W'\otimes U'$
with
$W' = W(n)\otimes (\otimes_{i\not\in I}S_{k_i}^{[j_i]})$ and
$U' = \otimes_{i\in I}S_{k_i}^{[j_i]}$.
It gives us $[V] = (\dim V)\mathcal{S}_{\alpha(V)} + r_V$ with
$\|r_V\| \leq C_r|W'|_{S,1}(\dim U')/(\min_{i\in I}k_i+1)$. 

We have $|W'|_{S,1} \leq M\|W'\|$ with $M$ as in Paragraph \ref{norms}.
As the norm $\|\cdot\|$ is an algebra norm, we get
$\|W' \| \leq \|W(n)\| \prod_{i\not\in I}\|S_{k_i}^{[j_i]}\|$.
Moreover $\|W(n)\| = \|W\|$  by
the properties of $\|\cdot\|$.

We apply Proposition \ref{frob} to each representation $S_{k_i}^{[j_i]}$ and we get an
inequality $\|S_{k_i}^{[j_i]}\| 
\leq (k_i+1)\|\mathcal{S}_{(\theta^{j_i}k_i \bmod {q-1})}\| + A$,
so that
$\|S_{k_i}^{[j_i]}\| \leq 2A(k_i+1)$ and
$|W'|_{S,1} \leq M\|W\|\prod_{i\not\in I}(k_i+1)$.

From $\dim U = 
(\dim U')\prod_{i\not\in I}(k_i+1)$ we get:
$$
\|r_V\| \leq MC_r(2A)^{h-r}\|W\|(\dim U)/(\min_{i\in I}k_i+1).
$$
Using the value of $C_r$ that we computed in the proof of
Theorem \ref{asymptote} we see that:
$$
\|r_V\| \leq Mq(2A+q)(2A)^{h}\|W\|(\dim U)/(\min_{i\in I}k_i+1).
$$
The best choice for $I$ is to take $I = \{i_0\}$ where $i_0$
is such that $k_{i_0} = \max k_i$, so that:
$$
\|r_V\| \leq Mq(2A+q)(2A)^{h}\|W\|(\dim U)/(\max_{1\leq i \leq h}k_i+1)
$$
In order have a formula that works for all $U$ independently of
the values of the $k_i$'s, the best inequality we
can use is $\max_ik_i+1 \geq (\dim U)^{1/h}$, which gives
the corollary with $C = Mq(2A+q)(2A)^{h} = MC_h$. 
\end{proof}

\subsection{Asymptotic values of multiplicities}

Let $a_{n,m}(V) = [V : L_n(m)]$. 
Our goal is to find the asymptotic value of $a_{n,m}(V)$.

We observe first:
\begin{lemm}
If $V$ has a central character $\alpha(V)$, 
then $a_{n,m}(V) = 0$ if $n + 2m \neq \alpha(V) \pmod {q-1}$.
\end{lemm}

Moreover Proposition \ref{calcul_omegan} implies:
\begin{prop}
\label{s_omegan}
\begin{enumerate}
\item
$[\mathcal{S}_i : L_n(m)] = 0$ if $n+2m \neq i \pmod {q-1}$
\item
$[\mathcal{S}_i : L_n(m)] = \omega(n)/(q^2-1)$ if $n+2m = i \pmod {q-1}$.
\end{enumerate}
\end{prop}

Finally we get
(the value of $C_r$ is given in Paragraph \ref{preuve} and the value of
$\omega(n)$ in Proposition \ref{calcul_omegan}):
\begin{theo}
Let $r \geq 1$ be an integer.
Let $W$ be a representation of $G$ with a central character.
Let $U = S_{k_1}^{[j_1]} \otimes \dots \otimes S_{k_r}^{[j_r]}$
and $V = W \otimes U$.
Then
$a_{n,m}(V) = 0$ if $n + 2m \neq \alpha(V) \pmod {q-1}$
and
$$
|a_{n,m}(V) - \omega(n)(\dim V)/(q^2-1)| \leq C_r|W|_{S,1}(\dim U)/\min_i(k_i+1)
$$
if $n+2m = \alpha(V) \pmod {q-1}$.
\end{theo}

\begin{proof}
Write $[V] = (\dim V)\mathcal{S}_{\alpha(V)} + r_V$
as in Theorem \ref{asymptote}. 
Then $$
|[V:L_n(m)] - (\dim V)[\mathcal{S}_{\alpha(V)}:L_n(m)]| \leq |[r_V:L_n(m)]|.
$$
By the conditions we imposed on the norm $\|\cdot\|$ in Paragraph
\ref{norms}, we have $|[r_V:L_n(m)]| \leq \|r_V\|$. The result follows.
\end{proof}

\begin{coro}
Let $W$ be a representation of $G$ with a central character, let $U$ be
a representation of $G$ that is the reduction in characteristic $p$
of an algebraic irreducible representation in characteristic $0$ of
$\operatorname{Res}_{K/\mathbb{Q}_p}\operatorname{GL}_2$ and let $V = W \otimes
U$. Then 
$$
|a_{n,m}(V) - \omega(n)(\dim V)/(q^2-1)| \leq C\|W\|(\dim
U)^{1-1/h}
$$
if $n+2m = \alpha(V) \pmod {q-1}$, and $a_{n,m}(V) = 0$ otherwise.
\end{coro}

\section{Application}
\label{section_application}

\subsection{The Breuil-Mézard conjecture}

Let us consider again the situation described in Paragraph
\ref{motivation} of the introduction. Let $K/\mathbb{Q}_p$ be a finite
extension and let $h = [K:\mathbb{Q}_p]$. We denote by $\kappa$ the residue
field of $K$ so that $\kappa = \mathbb{F}_q$ for some $q = p^f$.
Let $k$ be a finite field and let
$\overline{\rho} : \operatorname{Gal}(\overline{K}/K) \to \operatorname{GL}_2(k)$ be a Galois representation
such that $\operatorname{End}_{k[\operatorname{Gal}(\overline{K}/K)]}(\overline{\rho}) = k$.

Fix a Galois type $\mathtt{t}$, so that $\sigma(\mathtt{t})$ is an irreducible
finite-dimensional representation of $\operatorname{GL}_2(\mathcal{O}_K)$  
with coefficients in
$\overline{\mathbb{Q}}_p$. Fix also
$v = (n_{\tau},m_{\tau})_{\tau \in H_K}$
with $n_{\tau} \in \mathbb{Z}_{\geq 0}$ and 
$m_{\tau} \in \mathbb{Z}$. We denote by $\sigma(v)$ 
the representation
$\otimes_{\tau} (\operatorname{Symm}^{n_{\tau}}W_{\tau} \otimes \det_{\tau}^{m_{\tau}})$.

Let $\sigma(v,\mathtt{t}) = \sigma(\mathtt{t}) \otimes \sigma(v)$ 
and let 
$\overline{\sigma(v,\mathtt{t})}^{ss}$ be the semisimplification of 
the reduction to characteristic $p$ of 
$\sigma(v,\mathtt{t})$. It is a finite-dimensional
$k$-vector space with an action of $\operatorname{GL}_2(\mathcal{O}_K)$ that factors through
$\operatorname{GL}_2(\kappa)$. 
Note that we can 
write
$\overline{\sigma(v,\mathtt{t})}^{ss}$ as 
$(\overline{\sigma(\mathtt{t})}\otimes \otimes_{\tau} 
(\operatorname{Symm}^{n_{\tau}}E^2 \otimes \det{}_{\tau}^{m_{\tau}}))^{ss}$.

In order to compute the automorphic multiplicity,
which is defined by
$\mu_{Aut}(\overline{\rho},v,t) =
\sum_{\sigma}\mu_{\sigma}(\overline{\rho})a_{\sigma}(v,\mathtt{t})$,
we have to find the multiplicity
$a_{\sigma}(v,\mathtt{t})$ of $\sigma$ in the semi-simplification of
$\overline{\sigma(v,\mathtt{t})}$ for each irreducible representation
$\sigma$ of $\operatorname{GL}_2(\kappa)$.

For a fixed Galois type $\mathtt{t}$ and $v$ going to infinity (that is, with
the dimension of $\sigma(v)$ going to infinity) we can use the results of
Section \ref{section_asymptote} in order to have an asymptotic estimate 
of the values of the $a_{\sigma}(v,\mathtt{t})$ and therefore of
$\mu_{Aut}(\overline{\rho},v,\mathtt{t}) =
\Sigma_{\sigma}\mu_{\sigma}(\overline{\rho})a_{\sigma}(v,\mathtt{t})$, 
provided we know the intrinsic multiplicities
$\mu_{\sigma}(\overline{\rho})$. 

It is conjectured that 
$\mu_{\sigma}(\overline{\rho})$ can be nonzero only if $\sigma$
is an element of the set $D(\overline{\rho})$ 
of Serre weights defined in the article
\cite{BDJ} in the case where $K$ is unramified, and in 
\cite{Sch} and \cite{BLGG} in the general case. 
In particular, as all elements of $D(\overline{\rho})$
have the same central character, $\mu_{Aut}(\overline{\rho},v,\mathtt{t})$
is expected to be nonzero for only one possible value of the
central character $\alpha(\overline{\sigma(v,\mathtt{t})})$.

\subsection{Example: the case $K = \mathbb{Q}_p$}

In the case $K = \mathbb{Q}_p$ the Breuil-Mézard conjecture is entirely known thanks
to results of Kisin (\cite{Kisin1}) and Pa{\v{s}}k{\=u}nas (\cite{Pas}).
We will treat as an example the case where
$\overline{\rho} : G_{\mathbb{Q}_p} \to \operatorname{GL}_2(k)$ is an absolutely irreducible
representation. All the results we will use in this example can be found in
the article \cite{BrMez1}.

In this case,
$v = (a,b) \in(\mathbb{Z}_{\geq 0} \times \mathbb{Z})$ and
$\sigma(v) = \operatorname{Symm}^a \overline{\mathbb{Q}}_p^2\otimes \det^b$. 
The irreducible representations of $\operatorname{GL}_2(\mathbb{F}_p)$
are the $\sigma_{n,m} = \operatorname{Symm}^n E^2\otimes \det^m$
for $0 \leq n \leq p-1$ and $m \in \mathbb{Z}/(p-1)\mathbb{Z}$, and
$\sigma_{n,m}$ is the representation denoted by
$L_n(m)$ in Section \ref{section_asymptote}.
We denote $a_{\sigma_{n,m}}(v,\mathtt{t})$ by $a_{n,m}(v,\mathtt{t})$.
In this example we take for $\mathtt{t}$ the trivial type.

Let $\omega$ be the mod $p$ cyclotomic character and
$\omega_2$ the Serre fundamental character of level $2$,
so that $\omega_2^{p+1} = \omega$.

As $\overline{\rho}$ is irreducible we can write $\overline{\rho}_{|I_p}$
as
$$
\begin{pmatrix}{\omega_2^{n+1}} & 0 \\ 0 &
{\omega_2^{p(n+1)}}\end{pmatrix}\otimes\omega^m
$$
for some $0 \leq n \leq p-2$ and $m \in \mathbb{Z}/(p-1)\mathbb{Z}$. 
There are exactly two nonzero intrinsic multiplicities:
we have 
$\mu_{n,m}(\overline{\rho}) = \mu_{p-1-n,n+m}(\overline{\rho}) = 1$. 

As we have chosen $\mathtt{t}$ to be the trivial type, the representation
$\sigma(\mathtt{t})$ is the Steinberg representation
$\operatorname{Symm}^{p-1}\overline{\mathbb{Q}}_p^2$ and is of dimension $p$. 
The automorphic multiplicity $\mu_{Aut}(\overline{\rho},v,\mathtt{t})$
can be nonzero only for $v = (a,b)$ such that
the central character of
$\overline{\sigma(v,\mathtt{t})}$ is that same as the central character of 
those $\sigma_{n,m}$ for which $\mu_{n,m}(\overline{\rho}) \neq 0$,
that is 
$p+a+2b = n+2m \pmod {p-1}$. 

Finally we get, as $a \to +\infty$ and $p+a+2b = n+2m \pmod {p-1}$:
\begin{itemize}
\item
$\mu_{Aut}(\overline{\rho},(a,b),\mathtt{t}) = 4p(a+1)/(p^2-1) + O(1)$ 
if $1 \leq n \leq p-2$ (generic case)
\item
$\mu_{Aut}(\overline{\rho},(a,b),\mathtt{t}) = 2p(a+1)/(p^2-1) + O(1)$ if $n = 0$.
\end{itemize}

The constructions of Section \ref{motivation} can be generalized
to the case of deformation rings parameterizing representations that have
the additional property of being potentially crystalline. 
These rings differ from
the ones previously described only when the Galois type $\mathtt{t}$ is a scalar
(otherwise the condition of being potentially crystalline is automatically
satisfied). In this case, we need to replace the representation 
$\sigma(\mathtt{t})$ by another representation which we call
$\sigma^{\text{cr}}(\mathtt{t})$.  When $\mathtt{t}$ is trivial, 
$\sigma^{\text{cr}}(\mathtt{t})$ is the trivial representation.

In this case we get, when $\mathtt{t}$ is the trivial type
and as $a \to +\infty$ and $a+2b = n+2m \pmod {p-1}$:
\begin{itemize}
\item
$\mu^{\text{cr}}_{Aut}(\overline{\rho},(a,b),\mathtt{t}) = 4(a+1)/(p^2-1) + O(1)$ 
if $1 \leq n \leq p-2$.
\item
$\mu^{\text{cr}}_{Aut}(\overline{\rho},(a,b),\mathtt{t}) = 2(a+1)/(p^2-1) + O(1)$ if $n = 0$.
\end{itemize}

\subsection{Example: $K$ an unramified extension of $\mathbb{Q}_p$, and
$\overline{\rho}$ as generic as possible}
\label{ex_unram}

Let now $K$ be the unramified extension of degree $h$ of $\mathbb{Q}_p$.
In this situation we don't know if the Breuil-Mézard conjecture holds and we
don't even know which intrinsic multiplicities would make it true.
However, thanks to results of Gee and Kisin (\cite{GK}, Theorem A), we have some
results about the intrinsic multiplicities in the case where $K$ is
unramified. They show that $\mu_{\sigma}(\overline{\rho})$ is nonzero if and
only if $\sigma$ is an element of
the set $D(\overline{\rho})$ of Serre weights defined in the article
\cite{BDJ}. Moreover they show that if $\sigma$ is what they call
Fontaine-Laffaille regular (definition 2.1.8 in \cite{GK}), then
$\mu_{\sigma}(\overline{\rho}) = 1$ when it is nonzero. In our notations,
Fontaine-Laffaille regular weights are representations $L_n(m)$ where the
digits of $n$ in base $p$ are all in $\{0,\dots,p-3\}$. 

Let us fix a morphism $\mathbb{F}_{p^h} \to E$. We choose
a numbering $\{\tau_0,\dots,\tau_{h-1}\}$ 
of $H_K = \operatorname{Hom}(K,\overline{\mathbb{Q}}_p)$ such that
$\tau_i$ reduces modulo $p$ to the $i$-th power of Frobenius.
The algebraic irreducible representations of $\operatorname{Res}_{K/\mathbb{Q}_p}\operatorname{GL}_2$
are then the
$\sigma(v) = 
\otimes_{i=0}^{h-1}
(\operatorname{Symm}^{a_i}W_{\tau_i}\otimes \det_{\tau_i}^{b_i})$,
for $v = (a_i,b_i) \in (\mathbb{Z}_{\geq 0}^h \times \mathbb{Z}^h)$.
Such a representation reduces in characteristic $p$ to 
the representation
$S_{a_0}(b_0)^{[0]}\otimes
S_{a_1}(b_1)^{[1]}\dots 
\otimes S_{a_{h-1}}(b_{h-1})^{[h-1]}$.
Note that this representation has central character
$\sum_{i=0}^{h-1}p^i(a_i+2b_i)$.

The irreducible representations in characteristic $p$ of
$\operatorname{GL}_2(\mathbb{F}_q)$, $q = p^h$ are those described in Section
\ref{section_representations}
with $f = h$.
We denote $a_{L_n(m)}(v,\mathtt{t})$ by $a_{n,m}(v,\mathtt{t})$ 
as in Section \ref{section_asymptote}.

Let $\overline{\rho}$ a $2$-dimensional irreducible representation of $G_K$
with coefficients in $E$. 
After twisting $\overline{\rho}$ by some character we can write:
$$
\overline{\rho}_{|I} = 
\begin{pmatrix}
\omega_{2h}^{r_0+1+p(r_1+1)+\dots+p^{h-1}(r_{h-1}+1)} & 0 \\
0 & \omega_{2h}^{p^hr_0+1+p^{h+1}(r_1+1)+\dots+p^{2h}(r_{h-1}+1)}
\end{pmatrix}
$$
for some $0 \leq r_0 \leq p-1$ and some $-1 \leq r_i \leq p-2$
for $1 \leq i \leq h-1$ and where 
$\omega_{2h}$ denotes the Serre fundamental character of level $2h$.

In this example we take $\overline{\rho}$ to be as generic as possible,
that is we make the following assumption:
$$
2 \leq r_0 \leq p-3\text{ and }1\leq r_i \leq p-4\text{ for }i\geq 1.
$$

Note that in this case $\overline{\rho}$ is sufficiently generic 
in the sense of the article \cite{BP}, Lemma 11.4.
As a consequence, we see that 
$D(\overline{\rho})$ has exactly $2^h$ elements.
Moreover the elements of $D(\overline{\rho})$ can be described
using the method of \cite{BP}, Lemma 11.4. Using this description,
we see easily that our additional genericity condition ensures that
all $\sigma \in D(\overline{\rho})$ are Fontaine-Laffaille regular in the
sense of \cite{GK}, and thus have $\mu_{\sigma}(\overline{\rho}) = 1$.
This also ensures that for all $L_n(m) \in D(\overline{\rho})$, we
have $\omega(n) = 2^h$, as no digit of $n$ in base $p$ is $p-1$ and $n\neq 0$.
 
Note that all weights in $D(\overline{\rho})$ have central character
$r_0+1+p(r_1+1)+\dots+p^{h-1}(r_{h-1}+1)$.

Finally, under this hypothesis on the genericity of $\overline{\rho}$,
 we get, for a fixed Galois type $\mathtt{t}$,
$(a_i,b_i)$ such that 
$\alpha(\overline{\sigma(v,\mathtt{t})}) = \alpha(L_n(m))$ for all
$L_n(m)$ in $D(\overline{\rho})$,
that is
$$
\sum_{i=0}^{h-1}p^i(a_i+2b_i) +
\alpha(\overline{\sigma(\mathtt{t})}) = 
r_0+1+p(r_1+1)+\dots+p^{h-1}(r_{h-1}+1),
$$
and $\dim \sigma(v) \to +\infty$: 
$$
\mu_{Aut}(\overline{\rho},v,\mathtt{t}) = 
\frac{4^h}{p^{2h}-1}\dim \sigma(\mathtt{t})(\prod_{i=0}^{h-1}(a_i+1))
+ O(\prod_{i=0}^{h-1}(a_i+1))^{1-1/h})
$$


\begin{thebibliography}{BDJ10}

\bibitem[BDJ10]{BDJ}
Kevin Buzzard, Fred Diamond, and Frazer Jarvis.
\newblock On {S}erre's conjecture for mod {$\ell$} {G}alois
representations
  over totally real fields.
\newblock {\em Duke Math. J.}, 155(1):105--161, 2010.

\bibitem[BL94]{BL}
L.~Barthel and R.~Livn{\'e}.
\newblock Irreducible modular representations of {${\rm GL}_2$} of a
local
  field.
\newblock {\em Duke Math. J.}, 75(2):261--292, 1994.

\bibitem[BLGG]{BLGG}
T.~Barnet-Lamb, Gee, and D.~Geraghty.
\newblock Serre weights for rank two unitary groups.
\newblock To appear.

\bibitem[BM02]{BrMez1}
Christophe Breuil and Ariane M{{\'e}}zard.
\newblock Multiplicit{\'e}s modulaires et repr{\'e}sentations de {${\rm
  GL}_2({\bf Z}_p)$} et de {${\rm Gal}(\overline{\bf Q}_p/{\bf Q}_p)$} en
  {$l=p$}.
\newblock {\em Duke Math. J.}, 115(2):205--310, 2002.
\newblock With an appendix by Guy Henniart.

\bibitem[BM12]{BrMez2}
Christophe Breuil and Ariane M{{\'e}}zard.
\newblock Multiplicit{\'e}s modulaires raffin{\'e}es.
\newblock To appear, 2012.

\bibitem[Bon11]{Bon}
C{{\'e}}dric Bonnaf{{\'e}}.
\newblock {\em Representations of {${\rm SL}_2(\mathbb{F}_q)$}}, volume~13 of
{\em
  Algebra and Applications}.
\newblock Springer-Verlag London Ltd., London, 2011.

\bibitem[BP12]{BP}
Christophe Breuil and Vytautas Pa{\v{s}}k{\=u}nas.
\newblock Towards a modulo {$p$} {L}anglands correspondence for {${\rm
GL}_2$}.
\newblock {\em Mem. Amer. Math. Soc.}, 216(1016):vi+114, 2012.

\bibitem[Dia07]{Dia}
Fred Diamond.
\newblock A correspondence between representations of local {G}alois
groups and
  {L}ie-type groups.
\newblock In {\em {$L$}-functions and {G}alois representations}, volume
320 of
  {\em London Math. Soc. Lecture Note Ser.}, pages 187--206. Cambridge
Univ.
  Press, Cambridge, 2007.

\bibitem[DM12]{DM}
Agn{\`e}s David and Ariane M{\'e}zard.
\newblock Multiplicit{\'e} intrins{\`e}que des poids de {S}erre.
\newblock preprint, 2012.

\bibitem[EG11]{EmGee}
Matthew Emerton and Toby Gee.
\newblock A geometric perspective on the {B}reuil-{M}{\'e}zard
conjecture.
\newblock 09 2011.

\bibitem[GK12]{GK}
Toby Gee and Mark Kisin.
\newblock The {B}reuil–{M}{\'e}zard conjecture for potentially
  {B}arsotti–{T}ate representations.
\newblock Preprint, 2012.

\bibitem[Glo78]{Glo}
D.~J. Glover.
\newblock A study of certain modular representations.
\newblock {\em J. Algebra}, 51(2):425--475, 1978.

\bibitem[Kis08]{Kisin2}
Mark Kisin.
\newblock Potentially semi-stable deformation rings.
\newblock {\em J. Amer. Math. Soc.}, 21(2):513--546, 2008.

\bibitem[Kis09]{Kisin1}
Mark Kisin.
\newblock The {F}ontaine-{M}azur conjecture for {${\rm GL}_2$}.
\newblock {\em J. Amer. Math. Soc.}, 22(3):641--690, 2009.

\bibitem[Kis10]{KisinICM}
Mark Kisin.
\newblock The structure of potentially semi-stable deformation rings.
\newblock In {\em Proceedings of the {I}nternational {C}ongress of
  {M}athematicians. {V}olume {II}}, pages 294--311, New Delhi, 2010.
Hindustan
  Book Agency.

\bibitem[Pa{\v{s}}12]{Pas}
Vytautas Pa{\v{s}}k{\=u}nas.
\newblock On the {B}reuil-{M}{\'e}zard conjecture.
\newblock preprint, 2012.

\bibitem[Sch08]{Sch}
Michael~M. Schein.
\newblock Weights in {S}erre's conjecture for {H}ilbert modular forms:
the
  ramified case.
\newblock {\em Israel J. Math.}, 166:369--391, 2008.

\end{thebibliography}
\end{document}